\documentclass[12pt]{amsart}

\usepackage{mystyle}

\begin{document}

\title{Non-convex Hamilton-Jacobi equations with gradient constraints}

\author[H. A. Chang-Lara]{H\'ector A. Chang-Lara}
\address{Department of Mathematics, CIMAT, Guanajuato, Mexico}
\email{hector.chang@cimat.mx}

\author[E. A. Pimentel]{Edgard A. Pimentel}
\address{Department of Mathematics, Catholic University of Rio de Janeiro, Brazil
}
\email{pimentel@puc-rio.br}

\begin{abstract}
We study non-convex Hamilton-Jacobi equations in the presence of gradient constraints and produce new, optimal, regularity results for the solutions. A distinctive feature of those equations regards the existence of a \emph{lower bound} to the norm of the gradient; it competes with the elliptic operator governing the problem, affecting the regularity of the solutions. This class of models relates to various important questions and finds applications in several areas; of particular interest is the modeling of optimal dividends problems for multiple insurance companies in risk theory and singular stochastic control in reversible investment models. 
\end{abstract}

\subjclass{35B65; 35F21; 49N60; 49L25.}
\keywords{Nonconvex Hamilton-Jacobi equations; Lipschitz-continuity; optimal regularity.}

\maketitle


\section{Introduction}\label{sec:intro}

We study viscosity solutions to the non-convex Dirichlet problem
\begin{align}\label{main_bvp}
\begin{cases}
\min(-\D u-1,|Du|-1) = 0 \text{ in } \W\\
u = 0 \text{ on }, \p\W
\end{cases}
\end{align}
where $\W\subset \mathbb{R}^n$ is open and $\mathbb{R}^n\setminus\W$ is non-negligible for the Lebesgue measure. We establish interior Lipschitz-continuity for $u$ and prove that $|Du|$ is a continuous function.

The analysis of \eqref{main_bvp} is strongly motivated by a game-theoretic formulation we detail next. A token starting at $x\in \W$ moves until it exits the domain, according to choices made by two players in successive stages. Each turn takes place in a time interval with small length $dt>0$. The first player aims to maximize the time the particle spends in $\W$ whereas the second player wants to minimize this quantity. At the beginning of the time interval $[t,t+dt)$, the first player decides if the particle moves according to a Brownian motion (Brownian strategy) or with speed $1$ (eikonal strategy). For the latter case the second player chooses the direction $\theta\in\partial B_1$. 

The incentive for the first player to choose Brownian is in the possibility of getting the particle away from the boundary; this strategy is preferable close to a convex region of $\R^n \sm \W$. However, on average the Brownian motion displaces the token by a distance $dx \sim \sqrt{dt} \gg dt$, which is larger than the displacement obtained from constant speed. From the perspective of the second player, constant speed might be preferable, for it allows to control the direction of the displacement; meanwhile there is also an advantage from the Brownian strategy, as it yields a faster movement (with some chance) towards the boundary, see Figure \ref{fig1}.

\begin{center}
\begin{figure}
\includegraphics[scale=.7]{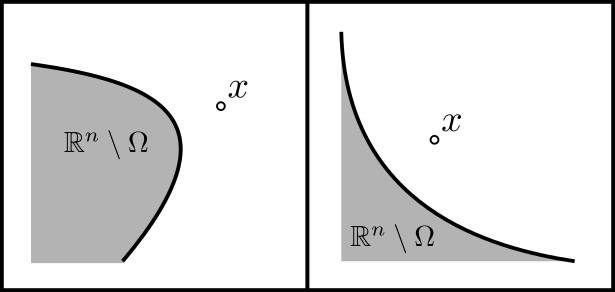}
\label{fig1}
\caption{In the scenario on the left, the first player may prefer to choose Brownian and get the particle away from the boundary with a good probability. In the scenario on the right it might be better to use eikonal to approach the boundary at a slower speed.}
\end{figure}
\end{center}

The classical optimal control theory approach suggests that in order to find the optimal strategies one needs to solve first a particular partial differential equation. Let $u(x)$ be the (optimal) value function of the game just described. It satisfies the following dynamic programming principle (DPP)
\[
u(x) = \max\1 \fint_{\p B_{\sqrt{2ndt}}}u(x+y)dy, \inf_{\theta\in \p B_1} u(x+\theta dt)\2 + dt.
\]
In the limit, as $dt\to 0^+$, this leads to the Hamilton-Jacobi-Bellman (HJB) in \eqref{main_bvp}.  For ease of notation, the diffusion in force in the case of the Brownian strategy is of the form
\[
	dx_t\,=\,\sqrt{2}\text{Id}\,dW_t,
\]
where $\text{Id}$ stands for the identity matrix in dimension $n$ and $W_t$ is a standard Brownian motion, adapted to a given stochastic basis. As a consequence, the associated equation is driven by the Laplacian, without the factor one-half. 

The importance of \eqref{main_bvp} is two-fold. First, it gives rise to a free boundary problem, as it splits $\W$ into two regions related to the optimal strategies for each player. Over $\{|Du|>1\}$ the first player chooses the Brownian strategy (we will refer to it as the Brownian region); over the region $\{|Du|=1\}$ (the eikonal region from now on) the first player chooses the eikonal strategy, and the second one moves in the direction of $-Du$. 

In addition to its connection with free boundary problems, \eqref{main_bvp} can also be classified as a Hamilton-Jacobi equation with gradient constraints. In fact, within the entire domain the gradient norm is bounded from below by $1$. We emphasize that, a priori, those bounds are understood in the viscosity sense. 

Hamilton-Jacobi equations in the presence of gradients constraints were first studied by Evans in \cite{MR529814}. In that paper, the author examines an equation of the form
\begin{equation}\label{eq_evans}
	\max\left\lbrace Lu-f,|Du|-g \right\rbrace\,=\,0 \text{ in } \W,
\end{equation}
where $L$ is an elliptic operator with twice-differentiable coefficients, and the functions $f, g\in C^2(\overline{\W})$ are given. Under those conditions, the author establishes existence of the solutions. Under the assumption that $L$ has constant coefficients, interior $C^{1,1}$-regularity for the solutions to \eqref{eq_evans} is available. In \cite{MR607553}, the author removes the assumption of constant coefficients and produces a $C^{1,1}$-interior regularity theory. A more general class of gradient constraints is considered in \cite{MR693645}; in that paper, the authors prove that solutions to \eqref{eq_evans} are $C^{1,1}$-regular. In addition, they further develop the uniqueness theory associated with the problem.

Further variants of the model in \eqref{eq_evans} have also been analyzed in the literature. We mention \cite{Yamada1988}, where the gradient constraint is paired with a finite family of elliptic operators $L_i$, with $i=1,\ldots,k$, and the maximum is taken with respect to the entire collection; see also \cite{Hynd-Mawi2016}, where the second order term is replaced with a fully nonlinear elliptic operator. We also refer the reader to \cite{MR3563780} and the references therein.

A further instance where gradient-constrained HJ equations appear is in the realm of singular optimal control problems. In this setting, as regards regularity theory, we refer the reader to \cite{MR1108426,MR1001925,MR1104105}. In those papers, the authors work under general conditions on the data of problem and prove that solutions are of class $C^2$, locally. Moreover, they examine the regularity of the associated free boundary. Once the regularity of the free boundary is available the authors are in position to construct the optimal control processes as reflected Brownian motions. For the analysis of the singular optimal control problem, we also mention \cite{Chow-Menaldi-Robin1985,Karatzas-Shreve1984,Karatzas-Shreve1985,Karatzas-Shreve1986}, just to name a few.


What makes a sharp distinction in our model is the type of restriction imposed on the gradient. In fact, \eqref{main_bvp} is non-convex with respect to the gradient. An immediate consequence relates to the gradient-constraint. While \cite{MR529814, MR3563780} imposes an \textit{upper} bound on $|Du|$, our problem imposes a \textit{lower} bound on it. It is not surprising that solutions to the type of problems in \cite{MR529814, MR3563780} are more regular than Lipschitz-continuous, whereas in our case Lipschitz regularity is indeed optimal.

The description of our problem can also be phrased by saying that an elliptic equation holds in the region \emph{where the gradient is large}. This broad class of degenerate elliptic equations in non-divergence form was studied by Imbert and Silvestre in \cite{MR3500837}. The main result in that work establishes that solutions to this class of problems are locally H\"older continuous, which follows from an Aleksandroff-Bakelman-Pucci type of estimate and a corresponding Harnack inequality. The strategy in that paper is to slide cusps from below, until they touch the graph of the solutions; it would enforce the equation to hold on contact points and produce a measure estimate. See \cite{MR3295593} for a generalization in the presence of merely bounded ingredients. When compared with our specific problem, our findings represent an improvement in the regularity -- all the way to Lipschitz-continuity -- and recover even stronger results for the norm of the gradient.

Further examples of operators with degenerate behavior depending on the gradient are the non-variational $p$-Laplacians. Even though there is a long list of results concerning these operators, we bring forward the work of Peres and Sheffield in \cite{MR2451291} as it is also motivated from a game-theoretical interpretation. Similar to the present work, the equation in \cite{MR2451291} defines the value function of a two-players game, which could be described as a stochastic tug-of-war with noise. The particular cases $p=1$ and $p=\8$ are perhaps the most interesting ones and were noticed before \cite{MR2451291}: $p=1$ relates with motion by mean curvature \cite{MR526057, MR2200259} and $p=\8$ stands for tug-of-war without noise \cite{MR2309980}. We would like to mention that the proof of Theorem \ref{thm:main} is inspired by the compactness approach of Wang for the $p$-Laplace equation in \cite{MR1264526}.

The original motivation for the study of \eqref{main_bvp} comes from an optimal dividends problems for multiple insurance companies in risk theory and mathematical finance. This type of models goes back to the work of Lundberg, Cr\'amer and de Finetti, developed in the first half of the 20th century, \cite{lundberg1909theorie, cramer, de1957impostazione}. We recommend the book by Azcue and Muler \cite{MR3287199} for a complete account of the theory, mainly developed by the probability community up to 2014.

Loosely adapting the model by Azcue, Muler, and Palmowski in \cite{MR3982209}, consider $n$ insurance companies, with accumulated surpluses at time $t$ given by a stochastic process
\[
X_t \in \W = \R^n_+ := \{x_i>0 \text{ for all } i\} 
\]
with initial condition $X_0 = x\in\W$. For instance,
\[
dX_t = b dt - dS_t
\]
where the drift $b \in \R^n_+$ is the constant income and $S_t = (\ldots, \sum_{i=1}^{N^{(j)}(t)} U^{(j)}_i , \ldots)$ describes the accumulated claims,  each component consisting of independent compounded Poisson processes with $U^{(j)}_i >0$.

Each branch uses its surplus to pay dividends, which is the control of the problem. Let $L(t) \in \overline{\R^n_+}$ be the accumulated dividends paid up to time $t$ such that $L(0)=0$ and each coordinate of $L$ is non-decreasing. Ruin occurs at the first time $\t$ at which $X-L$ leaves $\W = \R^n_+$. The goal is then to maximize the expected dividends paid by the branches until ruin, defining in this way the value
\[
u(x) = \sup_L \mathbb E_x \int_0^\t \sum_{i=1}^n dL_i(t).
\]
Other formulations may also include a discount factor in the previous formula.

Over a short interval of time, such that the process does not leave $\W$, we get the DPP
\[
u(x) = \sup_L\1\E_x u(X(dt)-L(dt)) + \sum_{i=1}^n L_i(dt)\2.
\]
Denote the infinitesimal generator of $X$ with $\mathcal{I}$, which is of integro-differential type in most applications. Then the HJB equation resulting from the DPP has the form
\[
\sup_{\ell \in \overline{\R^n_+}} \1 Iu - \sum_{i=1}^n \ell_i(\p_i u-1)\2 =0.
\]
The former equation entails natural constraints on the gradient of the solutions. Note that by taking $\ell=0$ we get $Iu\leq 0$; also, if $\p_i u-1<0$ for some coordinate, the supremum would be infinite. Hence, $-Iu, \p_1 u-1, \ldots, \p_n u-1\geq 0$ must be satisfied. Now it is easy to conclude that the previous equation builds upon the former discussion to yield
\[
\min(-Iu, H(Du)-1) = 0, \qquad H(Du) = \min(\p_1 u,\ldots, \p_n u)
\]
Notice that similar to our model, the main feature of this equation is that it enforces a lower bound on the gradient of the solution.

When analyzed in terms of the strategies that we develop on this paper, the main difficulty on this particular gradient-constraint problem seems to be the fact that $\{H\leq 1\}$ is an unbounded region. Otherwise we could say at least that there is an elliptic integro-differential equation which holds in the region where the gradient is large. We believe there are also technical issues to be addressed due to the non-local nature of the operator. However, these fall into the scope of the recent developments on the regularity theory of integro-differential elliptic equations; see for instance the survey by Ros-Oton in \cite{MR3447732}. Related formulations in singular optimal control, arising in the modeling of reversible investment policies and in the problem of purchasing electricity, are connected with Hamilton-Jacobi equations as in \eqref{main_bvp}; see \cite{Merhi-Zervos2007,Federico-Pham2014,DeAngelis-Ferrari-Moriarty2015}. 

In this work we analyze the existence, uniqueness and regularity of solutions to \eqref{main_bvp}. Existence and uniqueness follow by the comparison principle and Perron's method; this is reported in Section \ref{sec:pre}. In Section \ref{sec:reg} we establish an interior Lipschitz estimate for the solutions; see Theorem \ref{thm:lip}. We stress this is the optimal regularity for \eqref{main_bvp}, which follows from a series of examples put forward further in the paper. 

Our main result read as follows:

\begin{theorem}[Continuity of $|Du|$]\label{thm:main}
Let $\W \ss\R^n$ open, $r\geq 0$, and $u\in  C(\W)$ a viscosity solution to
\[
\min(-\D u-r,|Du|-1) = 0 \text{ in } \W.
\]
Then $|Du| \in C(\W)$.
\end{theorem}

We notice the equation in Theorem \ref{thm:main} is equivalent to \eqref{main_bvp}, due to a scaling argument (see Section \ref{sec:pre}). Finally, we notice the continuity of the gradient's norm enables us to prescribe, and make sense of, free boundary conditions of the form $|Du|=1$.

After analyzing the regularity of the solution, the next step would be to describe the interface, or free boundary, separating the two regimes. We believe it would be interesting to establish the semi-concavity of the solution, as it would unlock the rectifiability of the free boundary, as it has been done for the singular set of the eikonal equation by  Mantegazza and Mennucci in \cite{MR1941909}. We do not pursue this endeavor in the present paper.

The remainder of this paper is organized as follows. In Section \ref{sec:pre} we present some examples and detail results of general interest to the theory of viscosity solutions, such as stability, comparison principles, existence of solutions and global regularity. Lipschitz regularity is the subject of Section \ref{subsec_lipschitz}. We put forward the proof of Theorem \ref{thm:main} in Section \ref{subsec_mainresult}.

\bigskip

\noindent{\bf Acknowledgments:} The authors are grateful to Ovidiu Savin for his comments and suggestions on the material in this paper. HC acknowledges support from CONACyT-MEXICO Grant A1-S-48577. EP is partially supported by CNPq-Brazil (Grants \#433623/2018-7 and \#307500/2017-9), FAPERJ (Grant \#E.200.021-2018) and Instituto Serrapilheira (Grant \# 1811-25904). This study was financed in part by the Coordena\c{c}\~ao de Aperfei\c{c}oamento de Pessoal de N\'ivel Superior - Brazil (CAPES) - Finance Code 001.

\section{Preliminaries}\label{sec:pre}

In this section we examine further aspects of the problem in \eqref{main_bvp}. We proceed by discussing some examples and  present a few preliminary notions. Finally, we put forward results of general interest to the theory of viscosity solutions to \eqref{main_bvp}, including stability, comparison principles, existence and global regularity.

\subsection{Examples}\label{subsec_examples} Before proceeding we list some examples. Those account for distinctive aspects of the solutions to \eqref{main_bvp} and unveil important characteristics of the problem. The discussion is rather informal at this stage, the rigorous validation of these solutions is a consequence of our main theorems in the following sections.

\bigskip

\noindent\textbf{Example 1 (Purely eikonal regime).} For $n\geq 2$, let $\W = B_R \ss\R^n$, with $R \in (0,1/(n-1)]$. Then $u(x) = R-|x|$. An optimal strategy for the first player is to always choose the eikonal regime. In general, if the (inward) mean curvature of $\p \W$ is at least $1$ at every point, then the optimal strategy for the first player is always to choose the eikonal regime and $u$ is simply the distance function to $\R^n \sm \W$.

\bigskip

\noindent\textbf{Example 2 (Piecewise $C^1$-regular solutions).} For $n\geq 2$, if $\W = B_R \ss\R^n$ with $R > 1/(n-1)$ then
\[
u(x) = \begin{cases}
A+B\Phi(|x|)-|x|^2/(2n)& \text{ if } |x| \in (1/(n-1),R]\\
C-|x|& \text{ if } |x| \leq 1/(n-1),
\end{cases}
\]
where $\Phi$ is the fundamental solution, and the constants are related by the following conditions
\begin{align*}
&A+B\Phi(R) - R^2/(2n) = 0\\
&A+B\Phi(1/(n-1)) - 1/[2n(n-1)^2] = C - 1/(n-1)\\
&B\Phi'(1/(n-1))-1/[(n-1)n]=-1;
\end{align*}
the role of the former conditions is to ensure the solutions match the zero boundary value and glue pieces in a $C^1$ fashion.

\bigskip

\noindent\textbf{Example 3 (Approximating optimal strategies).} For $n=1$ and $\W = B_R= (-R,R)$ we have
\[
u(x) = (R+R^2/2) - (|x| + x^2/2).
\]
This example illustrates a singular type of behavior as it shows that there is not an exact strategy for the first player, but rather a sequence of them, approaching the optimal one. The idea is that for $\e>0$ small, the first player chooses eikonal only in the interval $(-\e,\e)$.

\bigskip

\noindent\textbf{Example 4 (Continuity of $|Du|$).} For $n\geq 2$, and $R>r>0$, set $\W = B_R\sm B_r$. Once again the solution is radial and there are two or three regimes, depending on how the radii $R$ and $r$ do compare with respect to $1/(n-1)$. If $R\in (0,1/(n-1)]$ we get that 
\[
u(x) = \begin{cases}
R-|x|& \text{ if } |x| \in (\r,R]\\
A+B\Phi(|x|)-|x|^2/(2n)& \text{ if } |x| \in [r,\r],
\end{cases}
\]
where $A$, $B$, and $\r \in (r,R)$ can be computed from
\begin{align*}
&A+B\Phi(r) - r^2/(2n) = 0\\
&A+B\Phi(\r) - \r^2/(2n) = R - \r\\
&B\Phi'(\r)=1.
\end{align*}
In this case, the pieces do not match in a $C^1$ fashion; meanwhile, we still conclude that $|Du|$ is a continuous function (in line with the main result in this paper).

If $R>1/(n-1)>r$ then $u$ solves the eikonal equation in $B_{1/(n-1)}\sm B_{\r}$, for some $\r\in(r,1/(n-1))$, and the Poisson equation in the remaining two rings. Finally if $r\geq 1/(n-1)$ we find two Brownian regimes. In these last two remaining cases, the parameters should be fixed in order to have continuous solutions with the free boundary condition $|Du|=1$ along the interfaces, see Figure \ref{fig2}.

\bigskip

\noindent\textbf{Example 5 (Infinite cost).} For $\W=\R^n\sm B_r$ the game-theoretic interpretation, or just the previous computation by letting $R\to\8$, suggests that the solution should be $u=+\8$. This is expected to happen in the case of a purely Brownian strategy, where the particle eventually hits $\overline{B_r}$, but in average this just takes an infinite amount of time.

\begin{center}
\begin{figure}
\includegraphics[scale=.7]{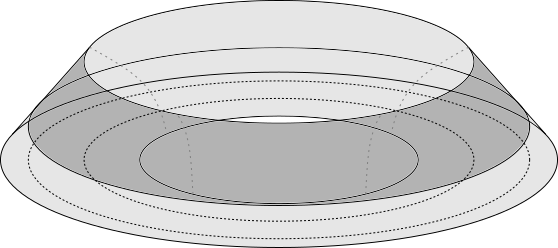}
\caption{Graph of the solution for $\W = B_R\sm B_r$ with $R>1/(n-1)>r$ exhibiting three regions, the middle (darker) one being eikonal.}
\label{fig2}
\end{figure}
\end{center}

Among the insights provided by the previous examples, we highlight the evidences that solutions to \eqref{main_bvp} should be no more regular than Lipschitz-continuous. We also anticipate that $|Du|$ is more than a merely bounded function. Those facts are in line with the results established in the present paper, Theorems \ref{thm:main} and \ref{thm:lip}. They also suggest the importance of the domain's geometry, as an ingredient affecting the existence of approximating optimal strategies and the possibility of infinite value functions.

\subsection{Scaling}

An important aspect of our analysis concerns the scaling properties of \eqref{main_bvp}. We start by noticing that both factors inside the minimum operator in \eqref{main_bvp} have different scaling regimes. Owing to the previous examples and the optimal regularity of solutions, we consider a Lipschitz-type of scaling. That is, for a solution $u$ of 
\[
	\min(-\D u-1,|Du|-1)=0\hspace{.4in}\mbox{in}\hspace{.1in} \W
\]
and $r>0$, we define
\[
v(x) = r^{-1}u(rx). 
\]
It is clear that 
\[
\min(-\D v-r,|Dv|-1)=0 \hspace{.4in}\text{ in }\hspace{.1in} r^{-1}\W
\]

For a game-theoretical interpretation consider now that the value of the game is computed with different weights in the two regimes. If the time in the Brownian regime is given by $r>0$ units of the eikonal time, then the corresponding equation becomes 
\[
\min(-\D u-r,|Du|-1)=0.
\]
Equivalently, one could keep counting time with equal weights but take the Brownian motion to be $1/r$ times faster. As a limit case when $r\to0$, we can even recover a scenario where time over the Brownian region does not count in the final value, giving us the equation
\[
\min(-\D u,|Du|-1)=0.
\]
Here the only incentive that the first player has to play the Brownian strategy is to keep the particle away from $\R^n\sm \W$. Let us compare the examples in the previous section for this problem:

\noindent\textbf{Examples 1, 2, and 3.} If $\W= B_R$, or in the more general case where the mean curvature of $\p\W$ is non-negative, we get that $u$ is the distance function to $\R^n\sm\W$. 

\bigskip

\noindent\textbf{Example 4.} For $n\geq 2$ and $\W = B_R\sm B_r$, there exists $\r \in (r,R)$ defining the Brownian region $B_\r\sm B_r$ whereas the remaining portion of $\W$ is the eikonal regime. The dimensional restriction entailed by the comparison of $R$ with $1/(n-1)$ no longer applies.

The last example might sound conflicting with the fact that the value function is positive on the Brownian region $B_\r\sm B_r$. In fact, suppose the particle starts at $x\in B_\r\sm B_r$; in this case, the second player simply allows the particle to move until it escapes $\W$ through $\p B_r$ with positive probability. Conversely, if the particles tries to escape through $\p B_\r$, then the second player pushes it back to $B_\r\sm B_r$, paying an infinitesimal amount of time. As the process evolves, those tiny contributions to the cost do in fact accumulate.

\bigskip

\noindent\textbf{Example 5.} $\W=\R^n\sm B_r$. For $n\geq 2$, $u=+\8$ is approached by a sequence of strategies where the first player chooses Brownian over $B_R\sm B_r$ with $R\to\8$. If $n=1$ instead, we get that $u=r-|x|$. In both cases the conclusion might result surprising once again, as the optimal strategy for the first player is to always choose the Brownian strategy, paying zero for the time, though the value turns out to be positive, or even infinite.

\bigskip

In the next section we recall the definition of viscosity solution used in the paper, and comment on its stability properties in the context of \eqref{main_bvp}.

\subsection{Viscosity solutions and stability properties}\label{subsec:pre}
 
Viscosity solutions are defined as in the classical literature. We use \cite{Crandall-Ishii-Lions1992} as our main reference in what follows.

\begin{definition}[Viscosity solution]\label{def_viscsol}
Given $u \in {\rm LSC}(\W)$ and $f \in C(\W)$, we say that it is a viscosity super-solution of
\[
\min(-\D u-r,|Du|-1) \geq f \text{ in } \W
\]
if for every $\varphi \in  C^\8(B_\r(x_0))$ such that $B_\r(x_0)\ss \W$ we have that
\[
\min_{B_\r(x_0)}(u-\varphi) = (u-\varphi)(x_0) = 0 \qquad\Rightarrow\qquad \min(-\D \varphi(x_0)-r,|D\varphi(x_0)|-1) \geq f(x_0)
\]
\end{definition}

Similarly, for a viscosity sub-solution of $\min(-\D u-r,|Du|-1) \leq f$ we require that whenever the test function touches $u \in  {\rm USC}(\W)$ from above, then it satisfies the corresponding inequality at the contact point. A viscosity solution is a continuous function which is simultaneously a sub and super-solution.

In the setup of the previous definitions, we say that $\varphi$ touches $u \in {\rm LSC}(\W)$ from below at $x_0$ over $B_\r(x_0)\ss\W$ whenever $\min_{B_\r(x_0)}(u-\varphi) = (u-\varphi)(x_0) = 0$, and a similar convention will be used when testing for super-solutions (from above).

Viscosity solutions are naturally stable under one-sided uniform convergence; see \cite{Crandall-Ishii-Lions1992}. As a reminder, and for further reference, we define the Gamma-convergence $u_k \overset{\G}{\to} u$ next. We say that $(u_k)_{k\in\mathbb{N}}$ Gamma converges to $u$, with $u_k,u:\W\to\mathbb{R}$, when the following conditions are satisfied: 
\begin{enumerate}
\item Whenever $x_k\to x$ in $\W$, we have 
\[
	u(x)\leq \liminf u_k(x_k);
\]
 
\item For every $x\in \W$, there exists a sequence $(x_k)_{k\in\mathbb{N}}$ such that $x_k\to x$ and 
\[
	\limsup u_k(x_k) \leq u(x).
\]
\end{enumerate}

Next, we observe that solutions to a variant of \eqref{main_bvp} are entitled to a stability result.

\begin{property}[Stability]
Let $(u_k)_{k\in\mathbb{N}}\ss {\rm LSC}(\W)$ and $(f_k)_{k\in\mathbb{N}} \ss C(\W)$ be a pair of sequences such that the following holds in viscosity 
\[
\min(-\D u_k-r,|Du_k|-1) \geq f_k \text{ in } \W.
\]
Then
\[
u_k \overset{\G}{\to} u \text{ and } f_k \overset{\text{unif.}}{\to} f \qquad \Rightarrow \qquad \min(-\D u-r,|Du|-1) \geq f.
\]
\end{property}


Thanks to the method of doubling the variables, we get that the comparison principle between sub-solutions and strict super-solutions holds; see, for instance \cite{Crandall-Ishii-Lions1992}. Here $u\in {\rm LSC}(\W)$ is a strict super-solution of $\min(-\D u-r,|Du|-1) \geq 0$ if there exists some $f>0$ such that $u$ is a super-solution of $\min(-\D u-r,|Du|-1) \geq f$.

Because of stability, the full comparison principle follows if we show that any super-solution can be approximated by strict super-solutions.

\begin{property}[Comparison Principle]
Let $r\geq 0$ and $\W \ss B_R$. Suppose $u,-v\in {\rm USC}(\overline\W)$ are such that the following hold in the viscosity sense:
\[
\min(-\D u-r,|Du|-1)\leq 0 \text{ in } \W \qquad\text{ and } \qquad \min(-\D v-r,|Dv|-1)\geq 0 \text{ in } \W.
\]
Then
\[
u\leq v \text{ on } \p\W \qquad \Rightarrow \qquad u\leq v \text{ in } \W.
\]
\end{property}

\begin{proof}
Assume without loss of generality that $v\geq 0$. Choose $0<\e,\d\ll 1$ small enough, such that 
\[
	w = (1+\e)v + \d(R^2-|x|^2),
\]
satisfies
\[
-\D w -r \geq 2n\d \qquad \text{ and } \qquad |Dw| - 1\geq \e - 2R\d 
\]
in the viscosity sense. By taking $\d = \e/(4R)$ we get that
\[
\min(-\D w-r,|Dw|-1) \geq (\e/2)\min(n/R,1)>0.
\]
The result now follows from the comparison principle with strict super solutions after taking $\e\to0$.
\end{proof}

\subsection{Dirichlet problem}

In order to construct solutions of \eqref{main_bvp} we resort to the Perron's method which depends on the comparison principle, which we have proved for bounded domains, and the existence of appropriated barriers.

\begin{property}[Boundary Value Problem]\label{prop_dirglobalreg}
Given $r\geq 0$, $\W\ss\R^n$ a bounded domain with an exterior cone condition, and $g\in C(\p\W)$, there exists a unique solution of
\[
\begin{cases}
\min(-\D u-r,|Du|-1) = 0 \text{ in } \W\\
u = g \text{ on } \p\W,
\end{cases}
\]
Moreover, if the exterior cone condition is uniform, then the solution has a modulus of continuity depending only on $r$, $\diam\W$, the modulus of continuity of $g$, and the parameters from the uniform exterior cone condition (radius and angle).
\end{property}

\begin{proof}
Given that the equation is invariant by vertical translations, let us assume that $\|g\|_{L^\8(\p\W)} = \osc_{\p\W}g$, so that quantities depending on $\|g\|_{L^\8(\p\W)}$ do in fact depend on the modulus of continuity of $g$.

For the construction of the barriers let us fix now $0\in \p \W$ and for some angle $\theta\in(0,\pi/2)$ and radius $\r>0$
\[
C = \{x \in \R^n \ | \ x_1 \geq |x|\cos \theta\} \qquad\text{ such that }\qquad C\cap B_\r \ss \R^n \sm \overline{\W}
\]
Let $h(x) = |x|^\a\phi(x/|x|)>0$ be the harmonic function in $\R^n \sm C$ with $h = 0$ on $\p C \sm \p B_\r$ and $h(-e_1) = 1$. Besides the harmonicity and the positivity of $h$, it is important to notice that $\inf_{\W}|Dh|>0$ because $\a \in (0,1)$ (given that $\theta\in(0,\pi/2)$). The importance of this last observation is that for some constant $C_0>0$ sufficiently large
\[ 
 \varphi(x) = C_0h(x) - (r/(2n))|x|^2
\]
is a super solution.

Let us now use this barrier to get a modulus of continuity at the origin. Given $\e>0$, let $\d_0>0$ such that $|g| <\e$ in $\p \W \cap B_{\d_0}$. Consider now the following candidates for upper ($+$) and lower ($-$) barriers
\[
 \varphi_\pm(x) = \pm(\e + C_0h(x)) - (r/(2n))|x|^2.
\]
For $C_0$ even larger we get that $|\varphi_\pm| \geq \|g\|_{L^\8(\p\W)}$ on $\p\W\sm B_{\d_0}$, so that by the comparison principle we get $\varphi_- \leq u \leq \varphi_+$. Finally, given that $\lim_{|x|\to 0} h(x) = 0$ there exists a $\d\in(0,\d_0)$ such that $|u| \leq 2\e$ in $\overline \W \cap B_\d$.

Up to this point we have shown that $u$ takes the boundary value $g$ and the existence of a modulus of continuity $\s$ for $u$ at boundary points (provided that the uniformity of the exterior cone condition). To get the global regularity we just use the comparison principle. Given $x_0,y_0\in \W$ let $\overline u(x) := u(x-(y_0-x_0))+\s(|y_0-x_0|)$ defined over $\widetilde \W = \overline\W + (y_0-x_0)$. It satisfies $\overline u\geq u$ over $\p(\widetilde \W\cap \W)$ by the barrier estimate. Thanks to the comparison principle this implies $\overline u(y_0) = u(x_0)+\s(|y_0-x_0|)\geq u(y_0)$ which is the desired inequality.
\end{proof}

\begin{remark}
 If $g \in C^{0,1}(\p\W)$ and $\W$ have a uniform exterior ball condition, a modification of the previous construction provides a Lipschitz modulus of continuity for $u$ in $\W$. In the next section we will show however that $u \in C^{0,1}_{loc}(\W)$ regardless of the modulus of continuity of $g$ and the regularity of $\W$.  
\end{remark}

\begin{remark}\label{remark_globapp}
When establishing local Lipschitz regularity (see Theorem \ref{thm:lip}) we apply Property \ref{prop_dirglobalreg} to an approximated problem. We notice Property \ref{prop_dirglobalreg} is available in this context as well, since the barriers in the approximated setting are the same as above and, moreover, they do not depend on the approximation parameter $\varepsilon>0$.
\end{remark}

\section{Interior Regularity}\label{sec:reg}

In this section we show the main regularity results in this paper. We start by considering an approximating problem. In fact, we modify our equation by adding a small viscosity term to the eikonal regime. It leads to 
\[
	\min(-\D u-r,-\varepsilon\D u+|Du|-1)=0 \text{ in } \W,
\]
which is equivalent to
\begin{equation}\label{eq_eq0}
	-\e\D u = \max(\e r,1-|Du|) \text{ in } \W.
\end{equation}
We will show that solutions to \eqref{eq_eq0} are $C^2$-approximations to the actual solutions of our original equation, with an interior H\"older estimate for $(|Du|-(1+\l))_+$, independent of $\e$, but unfortunately dependent on $\l>0$. The integrability of the second derivative will be used just as a technical device to rigorously get the sub-harmonicity of $(|Du|^2-1)_+$.

We notice that even smoother approximations of the equation could be obtained, using for instance
\[
	-\e \D u = \d\ln(e^{\e r/\d}+e^{1-|Du|^2/\d}) \text{ in } \W,
\]
instead of \eqref{eq_eq0}. However, for our arguments $C^2$-regularity of the approximating problem will be enough. Finally notice that as the results in this section are local, we can assume that the regularity provided from the approximation holds globally as part of the hypotheses.



\begin{proposition}[Convergence of the approximating problem]\label{prop_convapprox}
Let $(\e_k)_{k\in\mathbb{N}}\ss\mathbb{R}$ and $(u_k)_{k\in\mathbb{N}}\ss C(\W)$ be sequences such that 
\[
	-\e_k \D u_k = \max(\e_k r, 1-|D u_k|) \hspace{.4in} \mbox{ in }\hspace{.1in} \W,
\]
in the viscosity sense. Suppose $u_k\to u$ locally uniformly and $\e_k\to0$, as $k\to\infty$. Then $u$ is a viscosity solution of
\[
\min(\D u -r,|Du|-1)=0\hspace{.4in} \text{ in }\hspace{.1in} \W.
\]
\end{proposition}

\begin{proof}
Let $B_\r(x_0) \Subset \W$ and $\varphi \in C^\8(B_\r(x_0))$ be such that $\varphi$ strictly touches $u$ from below at $x_0$ over $B_\r(x_0)$. By the uniform convergence we deduce the existence of a sequence of real numbers $(c_k)_{k\in\mathbb{N}}\ss\mathbb{R}$ and a sequence $(x_k)_{k\in\mathbb{K}}$ in $\mathbb{R}^d$ such that $c_k\to0$, $x_k\to x_0$, and the test function $\varphi + c_k$ also touches $u_k$ from below at $x_k$ over $B_\r(x_0)$. From this fact we will deduce the two required inequalities to check that $u$ is a super-solution.

First, notice that 
\[
-\e_k \D\varphi(x_k) \geq \max(\e_k r,1-|D\varphi|(x_k)) \geq \e_k r,
\] 
which implies $-\D \varphi(x_0) \geq r$.

Secondly, we claim that $|D\varphi|(x_0) \geq 1$. Indeed, if on the contrary were $|D\varphi|(x_0) < 1-\theta$, for some $\theta\in (0,1)$, we would get $|D\varphi(x_k)| < 1-\theta$ and 
\[
\max(\e_k r,1-|D\varphi(x_k)|) \geq \theta,
\] 
for large values of $k\gg 1$. As a consequence, it would hold $-\e_k \D \varphi (x_k) \geq \theta$, which contradicts the smoothness of $\varphi$ around $x_0$.

Let now $B_\r(x_0) \Subset \W$ and $\varphi \in C^\8(B_\r(x_0))$ such that $\varphi$ strictly touches $u$ from above at $x_0$ over $B_\r(x_0)$. Assuming that $|D\varphi|(x_0)>1$, our goal is to show that $-\D\varphi(x_0)\leq0$. Again by the uniform convergence we get that for some $c_k\to 0$ and $x_k\to x_0$, the test function $\varphi + c_k$ also touches $u_k$ from above at $x_k$ over $B_\r(x_0)$. Moreover, if we assume $k$ sufficiently large we still have $|D\varphi|(x_k)>1$ such that $-\e_k\D \varphi(x_k) \leq \max(\e_k r, 1-|D\varphi|(x_k)) = \e_k r$; it gives the desired sub-solution inequality in the limit and concludes the proof.
\end{proof}

In the sequel, we study \eqref{eq_eq0} and resort to the compactness from Proposition \ref{prop_dirglobalreg}, the stability from Proposition \ref{prop_convapprox}, and the uniqueness of solutions to produce interior regularity estimates for our original problem.

\subsection{Lipschitz regularity}\label{subsec_lipschitz}

Given that viscosity solutions of \eqref{eq_eq0} are super-harmonic both in the viscosity and the distributional sense, we get the following $H^1$ interior estimate.\\

\begin{proposition}\label{prop_uh1}
Let $u\in C^2(\overline{B_1})$ be a solution to \eqref{eq_eq0}. Then there exists $C>0,$ independent of $\e>0$, such that
\[
	\left\|u\right\|_{H^1(B_{3/4})} \leq C\left\|u\right\|_{L^\infty(B_1)}.
\]
\end{proposition}

\begin{proof}
The function $v=\|u\|_{L^\infty(B_1)}-u$ is non-negative and sub-harmonic. By integrating $-\D v \leq 0$ against $\phi^2 v$, where $\phi\in C^\infty_0(B_1)$ is a non-negative test function with $\phi \equiv 1$ over $B_{3/4}$, we obtain
\[
\int_{B_1}|Dv|^2\phi^2 \leq \int_{B_1}v\phi |Dv||D\phi|\leq\ \frac{1}{2}\int_{B_1}|Dv|^2\phi^2+\frac{1}{2}\int_{B_1}|D\phi|^2v^2.
\]
which gives us the desired estimate after rearranging the terms.
\end{proof}

\begin{lemma}\label{lem:lip}
Let $u\in C^2(\overline{B_1})$ be a solution to \eqref{eq_eq0}. Then there exists $C>0,$ independent of $\e>0$, such that
\[
	\|Du\|_{L^\8(B_{1/2})} \leq C\left(\left\|u\right\|_{L^\infty(B_1)}+1\right).
\]
\end{lemma}

\begin{proof}
Thanks to the regularity for solutions of \eqref{eq_eq0} we get that $w =(|Du|^2-1)_+ \in H^1(B_{3/4})$ is sub-harmonic. Then by combining the mean value formula with the previous proposition we get
\[
\|Du\|^2_{L^\8(B_{1/2})} \leq \|w\|_{L^\8(B_{1/2})}+1 \leq C(\|Du\|_{L^2(B_{3/4})}^2+1) \leq C(\|u\|_{L^\8(B_1)}^2+1).
\]
\end{proof}

The main consequence of the previous results is that viscosity solutions to our original equation are locally Lipschitz-continuous. This is the subject of the next theorem.

\begin{theorem}\label{thm:lip}
Let $u\in  C(\W)$ be a viscosity solution to
\[
\min(-\D u-r,|Du|-1) = 0 \text{ in } \W.
\]
Then $u \in C^{0,1}_{\rm loc}(\W)$ and for every $\W' \Subset \W$ there exists $C>0$ such that
\[
	\|Du\|_{L^\8(\W')} \leq C\left(\left\|u\right\|_{L^\infty(\W)}+1\right),
\]
where $C=C(d, {\rm dist}(\W',\p\W))$.
\end{theorem}

\begin{proof}
By a standard localization and scaling argument it suffices to consider the problem over $B_1$ with $u \in C(\overline{B_1})$. Consider a family $(u_\e)_{\e>0}\ss C^2(\overline{B_1})$ of solutions to \eqref{eq_eq0} for $\e>0$ with boundary data $u_\e = u$ over $\p B_1$. 

This family is uniformly bounded and equicontinuous over $\overline{B_1}$, thanks to the comparison principle (cf. Property \ref{prop_dirglobalreg} and Remark \ref{remark_globapp}). Compactness and the uniqueness for the Dirichlet problem yields $u_\e \to u$, locally uniformly over $\overline{B_1}$. Therefore $u$ inherits the Lipschitz estimate from Lemma \ref{lem:lip}.
\end{proof}

Now that it has been established that viscosity solutions are Lipschitz regular, the following corollary is a straightforward consequence of the classical theory of first order Hamilton-Jacobi equations. See for instance Chapter 10 in \cite{MR2597943}.

\begin{corollary}
Let $u\in  C(\W)$ be a viscosity solution to
\[
\min(-\D u-r,|Du|-1) = 0 \text{ in } \W
\]
Then $|\{|Du|<1\}|=0$.
\end{corollary}

\subsection{Continuity of $|Du|$}\label{subsec_mainresult}



In this section we detail the proof of Theorem \ref{thm:main}. Our approach is inspired by the work in \cite{MR1264526} for the $p$-Laplace equation. We show that for any $\l\in(0,1)$ and $e\in\p B_1$, $(\p_e u-(1+\l))_+$ has an interior H\"older estimate depending on $\l$. The strategy consists on noticing that, depending on the size of $\{\p_e u \geq 1+\l\} \cap B_1$, either $u$ is flat or the oscillation of $(\p_e u-(1+\l))_+$ has to diminish.

We will keep working with the approximations $u \in C^2(B_1)$ for $\e \in(0,1)$, and assume the following hypotheses:
\begin{align}\label{hyp1}
\begin{cases}
-\e\D u \geq \max(0,1-|Du|) \text{ in } B_1\\
-\e\D u \leq \max(\e,1-|Du|) \text{ in } B_1
\end{cases}
\end{align}
These equations are now independent of the parameter $r$ and remain invariant by Lipschitz scalings (that zoom into smaller scales). We will also assume that
\begin{align}\label{hyp2}
(\p_e u-1)_+ \text{ is sub-harmonic.}
\end{align}
Finally, the re-normalization assumptions we fix from now on are
\begin{align}\label{hyp3}
u(0) = 0 \quad \text{ and } \quad \|Du\|_{L^\8(B_1)} \leq 1+M \quad \text{for some fixed } M\geq 1
\end{align}

\begin{remark}
Assumption \eqref{hyp2} seems to be the only technical obstacle that require us to work with the approximating solutions, instead of directly using the limit. The Lipschitz regularity gives only a weak convergence of the gradients, however it does not seem trivial to recover the sub-harmonicity of $(\p_e u-1)_+$ from this fact. This sub-harmonicity is only used in Lemma \ref{lem_osc1}.
\end{remark}

As a first step we show that if the Brownian region covers a large fraction of $B_1$, then $u$ is flat in a smaller scale.

\begin{lemma}[Improvement of flatness]\label{lem_flat}
Suppose \eqref{hyp1} and \eqref{hyp3}. For every $\d,\l \in (0,1)$ there exist $\eta,\r \in (0,1/2)$, independent of $\e \in (0,1)$, such that if
\[
	\sup_{e\in\p B_1}|\{\p_e u \geq 1+\l\}\cap B_1| \geq (1-\eta)|B_1|
\]
there is a linear function $L$ satisfying
\[
\|u-L\|_{L^\8(B_\r)} \leq \d \r,
\]
with $|DL| \in [1+\l,1+M]$.
\end{lemma}

\begin{proof}
We reason through a contradiction argument, so let us suppose that for some $\d_0 \in(0,1)$ the following holds: There is a sequence $(u_i)_{i\in\mathbb{N}}$ such that $u_i$ satisfies \eqref{hyp1}, \eqref{hyp3}, and for some $e\in \p B_1$
\[
\lim_{i\to \8} |\{\p_e u_i < 1+\l\}\cap B_1| = 0,
\]
but, for some $\r\in(0,1)$ to be fixed,
\begin{align}\label{contra}
	\|u_i-L\|_{L^\8(B_{\r})} > \d_0 \r
\end{align}
for any linear function $L$ with $|DL| \in  [1+\l,1+M]$.

In this setup we can assume without loss of generality that $u_i\to u \in C^{0,1}(\overline{B_1})$ uniformly and $|\{\p_e u < 1+\l\}\cap B_1| = 0$. Indeed, it suffices to show that $|\{\p_e u \leq \theta\}\cap B_1| = 0$ for any $\theta \in [0,1+\l)$. We can assume that $Du_i \to Du$ weakly star in $L^\8(B_1)$ such that using $\chi_{\{\p_e u\leq \theta\}\cap B_1}$ as a test function
\begin{align*}
\theta|\{\p_e u \leq \theta\}\cap B_1| &\geq \int_{\{\p_e u \leq \theta\}\cap B_1} \p_e u \\
&=\lim_{i\to\8} \int_{\{\p_e u \leq \theta\}\cap B_1} \p_e u_i \geq (1+\l)|\{\p_e u \leq \theta\}\cap B_1|.
\end{align*}
Given that $\theta\in(0,1+\l)$ was arbitrary, this implies $|\{\p_e u<1+\l\}\cap B_1| = 0$.

Let us show that $-\D u \in [0,1]$ in $B_1$. The stability tells us that the limit $u$ is super-harmonic (as a uniform limit of super-harmonic functions) and also satisfies the sub-solution inequality from $\min(-\D u - 1,|Du|-1)\leq 0$ (for a similar reason). We will use $|\{\p_e u <1+\l\}\cap B_1| = 0$ to show that whenever $\varphi$ is a test function that touches $u$ from above at $x_0 \in B_1$, then $|D\varphi|(x_0)>1$; so that $-\D \varphi(x_0) \leq 1$ as expected.

Let us consider a cone around $e$ with angle $\theta \in (0,\pi/2)$ to be fixed sufficiently small
\[
C = \{x\in \R^n: e\cdot x \geq |x|\cos \theta\}.
\]
Then, for $c = \int_{C\cap B_1} e\cdot x/|x|^n dx = \int_{C\cap \p B_1} e\cdot x dS(x)$
\begin{align*}
e\cdot D\varphi(x_0) &= \lim_{r\searrow 0} \frac{1}{cr}\int_{C\cap B_r}\frac{D\varphi(x+x_0)\cdot x}{|x|^n}dx\\
&= \lim_{r\searrow 0} \frac{1}{cr^n}\int_{C\cap \p B_r}(\varphi(x+x_0)-\varphi(x_0))dS(x)\\
&\geq \limsup_{r\searrow 0} \frac{1}{cr^n}\int_{C\cap \p B_r}(u(x+x_0)-u(x_0))dS(x)\\
&= \limsup_{r\searrow 0} \frac{1}{cr^n}\int_{C\cap B_r}\frac{Du(x+x_0)\cdot x}{|x|^n}dx\\
&\geq ((1+\l)\cos\theta -(1+M)\sin\theta)\frac{\int_{C\cap B_1}1/|x|^{n-1}dx}{\int_{C\cap B_1} e\cdot x/|x|^n dx}.
\end{align*}
The last expression can be made strictly larger than one by taking the angle $\theta$ sufficiently small.

Now that we have shown that $-\D u \in [0,1]$ in $B_1$ we get that $u \in C^{1,1/2}(B_{1/2})$ and $L(x) = Du(0)\cdot x$ satisfies $\|u-L\|_{L^\8(B_\r)} \leq C\r^{3/2}$ for every $\r$ sufficiently small and some $C>1$. Moreover $|DL| \in [1+\l,1+M]$ and by taking $\r = \d_0^{1/2}/C$ we get the desired contradiction to \eqref{contra} thanks to the uniform convergence of $u_i$ towards $u$.
\end{proof}

Now we consider the alternative case: If for all $e\in \p B_1$, the directional derivative $\p_e u$ gets below $1+\l$ in a positive fraction of $B_1$, then the oscillation of $(\p_e u-(1+\l))_+$ diminishes.

\begin{lemma}[Diminish of oscillation]\label{lem_osc1}
Suppose \eqref{hyp1}, \eqref{hyp2}, and \eqref{hyp3}. For every $\eta, \l \in (0,1)$ and $e\in \p B_1$, there exists $\s \in (0,1)$, independent of $\e$ and $e$, such that if
\[
|\{\p_e u < 1+\l\}\cap B_1| \geq \eta|B_1|,
\]
we get 
\[
\|(\p_e u-(1+\l))_+\|_{L^\8(B_{1/2})} \leq (1-\s)\|(\p_e u-(1+\l))_+\|_{L^\8(B_{1})}.
\]
%
%
\end{lemma}

\begin{proof}
Let $U=(\p_e u-(1+\l))_+$ and $N=\|U\|_{L^\8(B_{1})}$. Using that $N-U \in [0,N]$ is super-harmonic and $|\{N-U \geq N\}\cap B_1| \geq \eta|B_1|$ we get the desired improvement for $N-U$ from below thanks to the mean value formula.
\end{proof}

Putting together both alternatives, the conclusions of Lemma \ref{lem_flat} and Lemma \ref{lem_osc1} imply that either the oscillation of $(\p_e u-(1+\l))_+$ decreases or $u$ should be flat around a linear function. The following corollary states the result that we obtain after iterations.

\begin{corollary}\label{cor}
Suppose \eqref{hyp1}, \eqref{hyp2}, and \eqref{hyp3}. For every $\d,\l\in(0,1)$ and $e\in\p B_1$, there exist $\r,\s\in (0,1)$ and $C_0\geq 2$, independent of $\e>0$ and $e$, such that if $\p_e u(0) > 1+\l$, then
\[
k = \sup \3 i \in \N \ | \ \|(\p_e u-(1+\l))_+\|_{L^\8\1B_{\r^j}\2} \leq (1-\s)^jM \;\;\forall j \in \{0,1,\ldots,(i-1)\}\4 < \8
\]
and
\[
\|u-L\|_{L^\8\1B_{\r^{k}}\2} \leq \d \r^k
\]
for some linear function $L$, with $|DL| \in [1+\l,1+M]$.
\end{corollary}

\begin{proof}
We have that $k<\8$ because $\p_e u(0) > 1+\l$. Consider $v(x) = \r^{-(k-1)}u(\r^{k-1} x)$. One immediately checks that $v$ satisfies \eqref{hyp1}, \eqref{hyp2}, and \eqref{hyp3}. By the definition of $k$, the hypothesis of Lemma \ref{lem_osc1} can not hold, therefore Lemma \ref{lem_flat} applies instead. We conclude in this way that $\|u-L\|_{L^\8\1B_{\r^{i}}\2} \leq \d \r^i$ for some linear function $L$ with $|DL| \in [1+\l,1+M]$. 
\end{proof}

Now it is the moment to fix the flatness parameter $\d$. By the previous corollary with $r=\r^k$ and $\a = \ln(1-\s)/\ln\r$ we get that $u(x) = L(x) + r^{1+\a}v(x/r)$ can be seen as a perturbation from the linear function $L$. The perturbation $v$ can be made arbitrarily flat (i.e. $\|v\|_{L^\8(B_1)}\leq \d$) and satisfies the following equation in the viscosity sense
\[
-\D v \in [0,1] \text{ in } \{|Dv| < \l\}\cap B_1.
\]
That is to say that whenever $\varphi$ is a smooth test function that touches $v$ from above at some $x_0 \in B_1$ with $|D\varphi|(x_0) < \l$, then $-\D\varphi(x_0) \geq 0$. If on the other hand, $\varphi$ is a test function that touches $v$ from below at $x_0 \in B_1$ with $|D\varphi|(x_0) < \l$, then $-\D\varphi(x_0) \leq 1$.

This family of degenerate elliptic equations has nice regularity estimates under flatness hypotheses. In particular, we have the following $C^{1,\a}$ regularity estimate due to Savin in \cite[Theorem 1.3]{Savin2007}.

\begin{theorem}[\cite{Savin2007}]\label{savin}
For $\l>0$ there exist $\d,\a\in(0,1)$ and $C>0$, such that if $v \in C(B_1)$ is a viscosity solution of
 \[
-\D v \in [0,1] \text{ in } \{|Dv| < \l\}\cap B_1
 \]
with $\|v\|_{L^\8(B_1)}\leq \d$ then $v \in C^{1,\a}(B_{1/2})$ and
\[
 \|v\|_{C^{1,\a}(B_{1/2})}\leq C\d.
\]
\end{theorem}

\begin{remark}
 Theorem 1.3 in \cite{Savin2007} is actually a $C^{2,\a}$ regularity estimate which holds whenever $-\D v$ is a smooth function over $\{|Dv|<\l\}\cap B_1$. Our formulation is a consequence of the techniques developed in such paper. Otherwise, we could have kept track of the constant right-hand side (determined by the parameter $r$) in our equations and apply Theorem 1.3 as originally stated in \cite{Savin2007}. However, we preferred to drop the dependence on $r$ and keep the regularity result which matches the oscillation decay for $U$.  
\end{remark}

Combining this last theorem with Corollary \ref{cor} we finally get the expected modulus of continuity.

\begin{corollary}\label{coro2}
Suppose \eqref{hyp1}, \eqref{hyp2}, and \eqref{hyp3}. For every $\l\in(0,1)$ and $e\in\p B_1$, there exists $\a\in(0,1)$ and $C>0$, independent of $\e$ and $e$, such that $(\p_e u-(1+\l))_+ \in C^\a(B_{1/2})$ and
\[
\|(\p_e u-(1+\l))_+\|_{C^\a(B_{1/2})} \leq C
\]
In particular, $(|Du|-(1+\l))_+ = \sup_{e\in\p B_1} (\p_e u-(1+\l))_+$ has the same modulus of continuity.
\end{corollary}

\begin{proof}
It suffices to show that $(\p_e u-(1+\l))_+$ has an $\a$-H\"older modulus of continuity at the origin if $\p_e u(0) > 1+\l$, independently of $\e$ and $e$. The rest of the argument then follows from a covering argument, noticing that the continuity at the zero level set is just a consequence of the continuity over the positivity set. So let us assume $\p_e u(0) > 1+\l$ from now on.

Let $\d\in(0,1)$ sufficiently small, depending on $\l$, such that Theorem \ref{savin} applies. Let $\a\in(0,1)$ be the smallest exponent among the one in Corollary \ref{cor} ($= \ln(1-\s)/\ln\r$) and Theorem \ref{savin}. By Corollary \ref{cor} we get that
\[
\sup_{r\in (\r^k,1)} r^{-\a} \osc_{B_r} (\p_e u-(1+\l))_+ \leq C
\]
Finally Theorem \ref{savin} completes the modulus of continuity for $r\in (0,\r^k)$.
\end{proof}

\begin{proof}[Proof of Theorem \ref{thm:main}]
Let us assume without loss of generality that $r \in [0,1]$, otherwise we can zoom in the equation to enforce this condition. Let $\overline{B_\r(x_0)}\Subset \W$ with $\r\in(0,1)$, and consider $\bar u \in C^{0,1}(\overline{B_1})$ such that
\[
\bar u(x) = \frac{u(\r x+x_0)-u(x_0)}{\r}.
\]
This re-normalization satisfies \eqref{hyp3} for $M=(\|Du\|_{L^\8(B_\r(x_0))}-1)_+$, and solves $\min(-\D \bar u - \bar r, |D\bar u| - 1) = 0$ in $B_1$ for $\bar r = r\r \in [0,1]$.

Given $\e>0$, let $u_\e \in C^2(B_1)\cap C(\overline{B_1})$ be the solution to \eqref{eq_eq0} with $u_\e = \bar u$ over $\p B_1$. These solutions satisfy the main hypotheses \eqref{hyp1}, \eqref{hyp2}, \eqref{hyp3} and converge locally uniformly to $\bar u$ as $\e \to 0^+$.  Let us show that for any $\l>0$ and $e\in\p B_1$, $(\p_e u_\e-(1+\l))_+$ and $(|Du_\e|-(1+\l))_+$ also converge locally uniformly towards $(\p_e \bar u-(1+\l))_+$ and $(|D\bar  u|-(1+\l))_+$, up to a sub-sequence independent of $e$.

Thanks to the compactness provided by Corollary \ref{coro2}, we get that up to a sub-sequence, $(\p_e u_\e-(1+\l))_+$ and $(|Du_\e|-(1+\l))_+$ converge locally uniformly towards $U_e$ and $U$. For any $x_0 \in \{U>0\}\cap B_1$ we get that there exist some ball $B_\eta(x_0) \ss B_1$, some direction $e\in\p B_1$, and some angle $\theta \in (0,\pi/2)$, such that for any $e' \in \p B_1 \cap \{e\cdot e' \geq \cos\theta\}$
\[
\p_{e'} u_\e > 1+\l \text{ in } B_\eta(x_0).
\]
It means that the partial derivatives of $u_\e$ converge uniformly to the ones corresponding to $\bar u$ over this neighborhood of $x_0$.

If instead $x_0 \in \{U=0\}\cap B_1$ then for any $\d>0$ there exists some ball $B_\eta(x_0) \ss B_1$ such that $|Du_\e| < 1+\l+\d$ in $B_\eta(x_0)$. By the mean value theorem we get that the incremental quotient
\[
\frac{u_\e(x_0+he)- u_\e(x_0)}{h} < 1+\l+\d,
\]
for $h\in (0,\eta)$. Taking $\e\to0^+$, then $h\to 0^+$, and finally $\d\to 0^+$, we obtain
\[
\sup_{e\in \p B_1}\limsup_{h\to 0^+} \frac{\bar u(x_0+he)-\bar u(x_0)}{h}  \leq 1+\l.
\]
This concludes the argument to show that $(\p_e u_\e-(1+\l))_+$ and $(|Du_\e|-(1+\l))_+$ converge locally uniformly towards $(\p_e \bar u-(1+\l))_+$ and $(|D\bar u|-(1+\l))_+$. This finally implies the continuity of $|D\bar u|-1 = (|D\bar u|-1)_+$ around 
zero and concludes the proof of the theorem.
\end{proof}

\bibliographystyle{plain}
\bibliography{mybibliography}

\end{document}